\documentclass[11pt]{amsart}
\usepackage{tabularx,booktabs}
\usepackage{caption}
\usepackage{amsmath}
\usepackage{amsfonts}

\usepackage{amscd}
\usepackage{amsthm}
\usepackage{amssymb} \usepackage{latexsym}
\usepackage{eufrak}
\usepackage{euscript}
\usepackage{epsfig}
\usepackage{graphics}
\usepackage{array}
\usepackage{enumerate}
\usepackage{dsfont}
\usepackage{color}
\usepackage{wasysym}
\usepackage{hyperref}
\usepackage{pdfsync}

\newcommand{\bel}[1]{\begin{equation}\label{#1}}

\newcommand{\be}{\begin{equation}}

\newcommand{\ba}{\begin{eqnarray}}
\newcommand{\ea}{\end{eqnarray}}

\newcommand{\qe}{\end{equation}}

\newcommand{\Hmm}[1]{\leavevmode{\marginpar{\tiny%
$\hbox to 0mm{\hspace*{-0.5mm}$\leftarrow$\hss}%
\vcenter{\vrule depth 0.1mm height 0.1mm width \the\marginparwidth}%
\hbox to
0mm{\hss$\rightarrow$\hspace*{-0.5mm}}$\\\relax\raggedright #1}}}

\newtheorem{theorem}{Theorem}[section]

\newtheorem{lemma}[theorem]{Lemma}
\newtheorem{corollary}[theorem]{Corollary}

\makeatletter

\@addtoreset{equation}{section}

\makeatother
\begin{document}

\title[Liouville theorem for $V$-harmonic heat flows]{Liouville theorem for $V$-harmonic heat flows}

\author{Han Luo}
\address{Han Luo: School of Mathematics and Statistics, Nanjing University of Science and Technology, Nanjing
210094, Jiangsu, People’s Republic of China}
\email{\href{mailto:19110180023@fudan.edu.cn}{19110180023@fudan.edu.cn}}

\author{Weike Yu}
\address{Weike Yu: School of Mathematical Sciences, Ministry of Education Key Laboratory of NSLSCS, Nanjing Normal University, Nanjing, 210023, Jiangsu, People’s Republic of China}
\email{\href{mailto:wkyu2018@outlook.com}{wkyu2018@outlook.com}}

\author{Xi Zhang}
\address{Xi Zhang: School of Mathematics and Statistics, Nanjing University of Science and Technology, Nanjing
210094, Jiangsu, People’s Republic of China}
\email{\href{mailto:mathzx@ustc.edu.cn}{mathzx@ustc.edu.cn}}

\begin{abstract}
In this paper, we investigate $V$-harmonic heat flows from complete Riemannian manifolds with nonnegative Bakry-Emery Ricci curvature to complete Riemannian manifolds with sectional curvature bounded above. We give a gradient estimate of ancient solutions to this flow and establish a Liouville type theorem, which is a complementarity of \cite{MR4582134}.
\end{abstract}

\maketitle

\section{Introduction}\label{sec:intro}

In 1975, Yau \cite{{MR431040}} established the celebrated Liouville theorem which shows that on a complete Riemannian manifold of nonnegative Ricci curvature, there does not exist any nontrivial harmonic function bounded from one side. Later, Souplet–Zhang \cite{MR2285258} proved the following parabolic version of Yau’s Liouville theorem.
\begin{theorem}\label{thm:Souplet–Zhang}
Let $(M,g)$ be a complete Riemannian manifold with non-negative Ricci curvature. Then we have

(1) Let $u:M\times (-\infty, 0]\to (0,\infty)$ be a positive ancient solution to the heat equation. If
\begin{equation*}
u(x,t)=\text{exp}[o(d(x)+\sqrt{\vert t\vert})]
\end{equation*}    
near infinity, then $u$ must be a constant. Here $d(x)$ denotes the Riemannian distance
from a fixed point;

(2) Let $u:M\times (-\infty, 0]\to \mathbb{R}$ be an ancient solution to the heat equation. If
\begin{equation*}
u(x,t)=o(d(x)+\sqrt{\vert t\vert})
\end{equation*}    
near infinity, then $u$ is a constant.
\end{theorem}

Recently, Kunikawa-Sakurai \cite{MR4322556, MR4299896} have proved some analogue Liouville type theorems for an ancient super Ricci flow under a growth condition concerning reduced distance. One of their main results is stated as follows.
\begin{theorem}\label{thm:improved Choi Liouville Theorem}
Let $(M,g(t))_{t\in[0,\infty)}$ be an admissible, complete backward super Ricci flow. We assume 
\begin{equation*}
\mathcal{D}(U)\ge 0,\quad \mathcal{H}(U)\ge -\frac{\widetilde{S}}{t},\quad \widetilde{S}\ge 0  
\end{equation*}
for all vector fields $U$. Here $S=\frac{1}{2}\partial_{t} g$, $\widetilde{S}=\text{tr}\,S$, $\mathcal{D}(U)=-\partial_{t} \widetilde{S}-\Delta \widetilde{S}-2\Vert S\Vert^2+4\text{div}\, S(U)-2g(\nabla \widetilde{S},U)+2\text{Ric}(U,U)-2S(U,U)$ and
$\mathcal{H}(U)=-\partial_{t} \widetilde{S}-\frac{1}{t}\widetilde{S}-2g(\nabla\widetilde{S},U)+2S(U,U)$. Let $(N,h)$ be a complete Riemannian manifold with sectional curvature bounded above by a positive constant $\kappa$. Assume that an open geodesic ball $B_{\pi/2\sqrt{\kappa}}(y_0)$ of radius $\pi/2\sqrt{\kappa}$ centered at $y_0$ in $N$ does not meet the cut locus $\text{Cut}(y_0)$ of $y_0$. Suppose that $u:M\times [0,\infty)\to N$ is a solution to backward harmonic map heat flow. If the image of $u$ is contained in $B_{\pi/2\sqrt{\kappa}}(y_0)$ and $u$ satisfies a growth condition 
\begin{equation*}
\frac{1}{\cos\sqrt{\kappa}\rho(u(x,t))} = o(\sigma(x,t)^{1/2}+t^{1/4})   
\end{equation*}
near infinity, then $u$ is constant. Here $\rho:N\to \mathbb{R}$ is the Riemannian distance function from a fixed point $y_0\in N$.
\end{theorem}

Furthermore, using the above theorem, Kunikawa-Sakurai improved the Liouville theorem of Choi \cite{MR647905} as follows: 
\begin{theorem}\label{thm:improved Choi Liouville Theorem}
Let $(M,g)$ be a complete Riemannian manifold with nonnegative Ricci curvature, and let $(N,h)$ be a complete Riemannian manifold with sectional curvature bounded above by a positive constant $\kappa$. Assume that $B_{\pi/2\sqrt{\kappa}}(y_0)$ does not meet $\text{Cut}(y_0)$. Let $u:M\to N$ be a harmonic map. If the image of $u$ is contained in $B_{\pi/2\sqrt{\kappa}}(y_0)$ and $u$ satisfies a growth condition 
\begin{equation*}
\frac{1}{\cos\sqrt{\kappa}\rho(u(x))} = o(d(x)^{1/2})   
\end{equation*}
near infinity, then $u$ is constant.
\end{theorem}

From now on, we suppose that $(M, g)$ and $(N, h)$ are both complete Riemannian manifolds. Recall that a map $u: M\to N$ is called a $V$-harmonic map if it solves
\begin{equation}\label{eq:V-harmonic map}
\tau_{V}(u):=\tau(u)+du(V)=0.
\end{equation}
According to \cite{MR3427131}, we have learnt that some interesting maps are exactly $V$-harmonic maps, such as Hermitian harmonic maps(c.f. \cite{MR1226528, MR2296630, MR3077216}), Wely harmonic maps(c.f. \cite{MR2520354}), affine harmonic maps(c.f. \cite{MR2554637, MR2796654}) and harmonic maps from Finsler manifolds into Riemannian manifolds(c.f. \cite{MR1757888, MR2407111, MR2483812, MR1895460, MR2156621}). Moreover, the results and methods of $V$-harmonic maps have been applied successfully in submanifolds geometry and mean curvature flows(c.f. \cite{MR3479571, MR4478479, MR4496501, MR4290282}), since Gauss maps of a self-shrinker or a translating soliton can be seen as a $V$-harmonic map. 

The heat flow associated to (\ref{eq:V-harmonic map}) is
\begin{equation}\label{eq:V-harmonic heat flow}
\frac{\partial}{\partial t}u=\tau_{V}(u)=\tau(u)+du(V).
\end{equation}
Chen-Jost-Wang \cite{MR3427131} proved the global existence and convergence of the $V$-harmonic map heat flow from compact manifolds into regular balls by using the maximum principle and continuous method. In \cite{MR2995205}, Chen-Jost-Qiu gave the existence theorems for $V$-harmonic maps between complete manifolds. When the source manifold is a complete noncompact Riemmanian manifold with nonnegative Bakry–Emery Ricci curvature, Chen-Qiu \cite{MR4582134} obtained Souplet–Zhang type and Choi type Liouville theorems of ancient solutions to (\ref{eq:V-harmonic heat flow}) under certain growth condition near infinity. Inspired by the above works, this paper aims to establish the analogue of the Liouville theorem for ancient solutions to (\ref{eq:V-harmonic heat flow}). 

Suppose the vector field $V$ satisfies
\begin{equation}\label{eq:additional growth condition 1}
 g(V,\nabla r)\le v(r)   
\end{equation}
for some nondecreasing function $v(\cdot)$ satisfying
\begin{equation}\label{eq:additional growth condition 2}
\lim_{r\to\infty} \frac{\vert v(r) \vert}{r}=C_V<+\infty,
\end{equation} 
where $r$ is the Riemannian distance of $(M,g)$ from a fixed point. Our main results are as follows:

\begin{theorem}\label{thm:Liouville Theorem heat flow}
Let $(M,g)$ be a complete Riemannian manifold with nonnegative Bakry-Emery Ricci curvature, and $(N,h)$ be a complete Riemannian manifold with sectional curvature bounded above by a positive constant $\kappa$. Assume that $B_{\pi/2\sqrt{\kappa}}(y_0)$ does not meet $\text{Cut}(y_0)$. Suppose that the vector field $V$ satisfies (\ref{eq:additional growth condition 1}) for some nondecreasing function $v(\cdot)$ satisfying (\ref{eq:additional growth condition 2}). Let $u:M\times(-\infty, 0]\to N$ be an ancient solution to (\ref{eq:V-harmonic heat flow}). If the image
of $u$ is contained in $B_{\pi/2\sqrt{\kappa}}(y_0)$, and if $u$ satisfies a growth condition
\begin{equation*}
\frac{1}{\cos\sqrt{\kappa}\rho(u(x))} = o(r(x)^{1/2}+|t|^{1/4})   
\end{equation*}
near infinity, then $u$ is constant.
\end{theorem}

A direct application of Theorem \ref{thm:Liouville Theorem heat flow} is the following result.

\begin{corollary}\label{thm:Liouville Theorem}
Let $(M,g)$ be a complete Riemannian manifold with nonnegative Bakry-Emery Ricci curvature, and $(N,h)$ be a complete Riemannian manifold with sectional curvature bounded above by a positive constant $\kappa$. Assume that $B_{\pi/2\sqrt{\kappa}}(y_0)$ does not meet $\text{Cut}(y_0)$. Suppose that the vector field $V$ satisfies (\ref{eq:additional growth condition 1}) for some nondecreasing function $v(\cdot)$ satisfying (\ref{eq:additional growth condition 2}). Let $u:M\to N$ be a V-harmonic map. If the image of $u$ is contained in $B_{\pi/2\sqrt{\kappa}}(y_0)$, and if $u$ satisfies a growth condition
\begin{equation*}
\frac{1}{\cos\sqrt{\kappa}\rho(u(x))} = o(d(x)^{1/2})   
\end{equation*}
near infinity, then $u$ is constant.
\end{corollary}

\section{Proof of the main theorem}\label{sec:proof}
Before giving our proof of main results, we prove a lemma we need. 
For $R,T>0$, we define 
\begin{equation*}
\begin{split}
&B_R(x_{0})=\{x\in M\vert r(x)< R\}, \\
&Q_{R,T}=\overline{B_R(x_0)}\times[-T, 0]\subset M\times(-\infty,+\infty).    
\end{split}
\end{equation*}
Here $r(x)$ is the Riemannian distance function from a fixed point $x_0$ on $M$.
As in \cite[Theorem 5.1]{MR4299896}, we set
\begin{equation}\label{2.1.}
\varphi=1-\cos\sqrt{\kappa}\rho, \qquad b=\frac{1}{2}\left(1+\underset{Q_{R,T}}{\text{sup}}\varphi\circ u\right)
\end{equation}
and 
\begin{equation*}
\omega=\frac{\Vert du\Vert^2}{(b-\varphi\circ u)^2},
\end{equation*}
where $u:M\times(-\infty, 0]\to N$ is an ancient solution of \eqref{eq:V-harmonic heat flow}, then we have the following: 

\begin{lemma}\label{thm:estimate of u}
Let $(M,g)$ be a complete Riemannian manifold with Barky-Emery Ricci curvature 
\begin{equation*}
\text{Ric}_V=\text{Ric}-\frac{1}{2}L_V g\ge -A
\end{equation*}
where $\text{Ric}$ is the Ricci curvature of $M$, $L_V$ is the Lie derivative along the vector field $V$ on $M$, and $A$ is a nonnegative constant. Let $(N,h)$ be a Riemannian manifold with sectional curvature bounded above by a constant $\kappa>0$. Assume that $B_{\pi/2\sqrt{\kappa}}(y_0)$ does not meet the cut locus $\text{Cut}(y_0)$ of $y_0\in N$. Let $u:M\times(-\infty, 0]\to N$ be an ancient solution to $V$-harmonic map heat flow \eqref{eq:V-harmonic heat flow}. Suppose that the image of $u$ is contained in $B_{\pi/2\sqrt{\kappa}}(y_0)$.
Then we have
\begin{equation*}
\begin{split}
(\Delta_V-\partial_t)\omega-\frac{2g(\nabla\omega, \nabla(\varphi\circ u))}{b-\varphi\circ u} 
\ge & 2\kappa(1-b)(b-\varphi\circ u)\omega^2  
-2A\omega,    
\end{split}
\end{equation*}
where $\Delta_V=\Delta+g(V,\nabla\cdot)$ is the $V$-Laplace operator.
\end{lemma}
\begin{proof}
Straightforward computations give that
\begin{equation*}
\begin{split}
\nabla\omega = &\frac{\nabla\Vert du\Vert^2}{(b-\varphi\circ u)^2}+2\frac{\Vert du\Vert^2\nabla(\varphi\circ u)}{(b-\varphi\circ u)^3}, \\
\Delta_V\omega = &\frac{\Delta_V\Vert du\Vert^2}{(b-\varphi\circ u)^2}+\frac{4g(\nabla\Vert du\Vert^2, \nabla(\varphi\circ u))}{(b-\varphi\circ u)^3}+\frac{2\Vert du\Vert^2\Delta_V(\varphi\circ u)}{(b-\varphi\circ u)^3} \\
&+\frac{6\Vert\nabla(\varphi\circ u)\Vert^2\Vert du\Vert^2}{(b-\varphi\circ u)^4} \\
= & \frac{2g(\nabla\omega, \nabla(\varphi\circ u))}{b-\varphi\circ u}+\frac{2\Vert du\Vert^2\Delta_V(\varphi\circ u)}{(b-\varphi\circ u)^3}+\frac{\Delta_V\Vert du\Vert^2}{(b-\varphi\circ u)^2} \\
&+\frac{2g(\nabla\Vert du\Vert^2, \nabla(\varphi\circ u))}{(b-\varphi\circ u)^3}+\frac{2\Vert\nabla(\varphi\circ u)\Vert^2\Vert du\Vert^2}{(b-\varphi\circ u)^4},\\
\partial_t \omega= &\frac{\partial_t\Vert du\Vert^2}{(b-\varphi\circ u)^2}+\frac{2\Vert du\Vert^2 \partial_t(\varphi\circ u)}{(b-\varphi\circ u)^3}.
\end{split}
\end{equation*} 
Then we get
\begin{equation*}
\begin{split}
(\Delta_V-\partial_t)\omega-\frac{2g(\nabla\omega, \nabla(\varphi\circ u))}{b-\varphi\circ u} =&\frac{(\Delta_V-\partial_t)\Vert du\Vert^2}{(b-\varphi\circ u)^2}+\frac{2\Vert du\Vert^2(\Delta_V-\partial_t)(\varphi\circ u)}{(b-\varphi\circ u)^3}  \\
&+\frac{2g(\nabla \Vert du\Vert^2,\nabla(\varphi\circ u))}{(b-\varphi\circ u)^3}+\frac{2\Vert\nabla(\varphi\circ u)\Vert^2\Vert du\Vert^2}{(b-\varphi\circ u)^4}. 
\end{split}
\end{equation*}
Due to Hessian comparison Theorem\cite{MR4299896},
\begin{equation*}
\begin{split}
(\Delta_V-\partial_{t})(\varphi\circ u) &= \sum\limits_{i=1}^m \nabla^2\varphi(du(e_i),du(e_i))+ \langle \nabla \varphi\circ u, \tau_{V}(u)-\partial_{t} u\rangle_h \\
&= \sum\limits_{i=1}^m \nabla^2\varphi(du(e_i),du(e_i)) \\
&\ge \kappa \cos \sqrt{\kappa} \rho \circ u\Vert du\Vert^2.
\end{split}
\end{equation*} 
The parabolic Bochner formula in \cite[Lemma 1]{MR3611336} tells us that
\begin{equation*}
\begin{split}
(\Delta_V-\partial_{t})\Vert du\Vert^2 =& 2\Vert\nabla du\Vert^2 + 2\sum\limits_{i=1}^m \langle du(\text{Ric}_V(e_i)), du(e_i)\rangle \\
&- 2\sum\limits_{i,j=1}^m R^N(du(e_i), du(e_j), du(e_i), du(e_j)) \\
\ge & 2\Vert \nabla du \Vert^2 - 2A \Vert du\Vert^2 -2\kappa \Vert du\Vert^4. 
\end{split}
\end{equation*} 
Hence,
\begin{equation*}
\begin{split}
(\Delta_V-\partial_t)\omega-\frac{2g(\nabla\omega, \nabla(\varphi\circ u))}{b-\varphi\circ u} =&\frac{(\Delta_V-\partial_t)\Vert du\Vert^2}{(b-\varphi\circ u)^2}+\frac{2\Vert du\Vert^2(\Delta_V-\partial_t)(\varphi\circ u)}{(b-\varphi\circ u)^3}  \\
&+\frac{2g(\nabla \Vert du\Vert^2,\nabla(\varphi\circ u))}{(b-\varphi\circ u)^3}+\frac{2\Vert\nabla(\varphi\circ u)\Vert^2\Vert du\Vert^2}{(b-\varphi\circ u)^4} \\
\ge &\frac{1}{(b-\varphi\circ u)^2}(\Vert \nabla du \Vert^2  - 2A \Vert du\Vert^2-2\kappa \Vert du\Vert^4) \\
&+\frac{\kappa \cos \sqrt{\kappa} \rho \circ u\Vert du\Vert^4}{(b-\varphi\circ u)^3}+\frac{2g(\nabla \Vert du\Vert^2,\nabla(\varphi\circ u))}{(b-\varphi\circ u)^3}\\
&+\frac{2\Vert\nabla(\varphi\circ u)\Vert^2\Vert du\Vert^2}{(b-\varphi\circ u)^4} \\
\ge &2\kappa(b-\varphi\circ u)^2(\frac{\cos\sqrt{\kappa}\rho}{b-\varphi\circ u}-1)\omega^2 \\
& -\frac{2A\Vert du\Vert^2}{(b-\varphi\circ u)^2}+2\mathcal{F}  \\
=& 2\kappa(1-b)(b-\varphi\circ u)\omega^2 -\frac{2A\Vert du\Vert^2}{(b-\varphi\circ u)^2} +2\mathcal{F},
\end{split}
\end{equation*}
where 
\begin{equation*}
\mathcal{F}=\frac{\Vert\nabla du\Vert^2}{(b-\varphi\circ u)^2}+\frac{g(\nabla \Vert du\Vert^2,\nabla(\varphi\circ u))}{(b-\varphi\circ u)^3}+\frac{\Vert\nabla(\varphi\circ u)\Vert^2\Vert du\Vert^2}{(b-\varphi\circ u)^4}.
\end{equation*}
By the inequality of arithmetic-geometric means, Kato and Cauchy-Schwarz inequalities, we have
\begin{equation*}
\begin{split}
\frac{\Vert\nabla du\Vert^2}{(b-\varphi\circ u)^2}+\frac{\Vert\nabla(\varphi\circ u)\Vert^2\Vert du\Vert^2}{(b-\varphi\circ u)^4} &\ge \frac{2\Vert \nabla du\Vert\Vert\nabla(\varphi\circ u)\Vert\Vert du\Vert}{(b-\varphi\circ u)^3}\\
& \ge \frac{\Vert\nabla \Vert du\Vert^2\Vert\Vert\nabla(\varphi\circ u)\Vert }{(b-\varphi\circ u)^3}\\
& \ge \frac{\vert g(\nabla \Vert du\Vert^2,\nabla(\varphi\circ u))\vert}{(b-\varphi\circ u)^3},
\end{split}
\end{equation*}
so $\mathcal{F}\geq 0$, which ends the proof.
\end{proof}

\begin{theorem}\label{thm:gradient estimate}
Let $(M^m,g)$ be a complete Riemannian manifold with Barky-Emery Ricci curvature
\begin{equation*}
\text{Ric}_V\ge -A
\end{equation*}
where $A\geq0$ is a constant. Let $(N,h)$ be a complete Riemannian manifold with sectional curvature bounded above by a positive
constant $\kappa$. Assume that $B_{\pi/2\sqrt{\kappa}}(y_0)$ does not meet $\text{Cut}(y_0)$. Suppose that the vector field $V$ satisfies (\ref{eq:additional growth condition 1}) for some nondecreasing function $v(\cdot)$ satisfying (\ref{eq:additional growth condition 2}). Let $u: M\times(-\infty, 0]\to N$ be an ancient solution to the V-harmonic map heat flow \eqref{eq:V-harmonic heat flow}. Suppose that the image of $u$ is contained in $B_{\pi/2\sqrt{\kappa}}(y_0)$. 
Then there is a positive constant $C_m>0$ depending only on $m$ such that for any $R,T>0$,
\begin{equation*}
\sup_{Q_{R/2,T/4}}\frac{\Vert du\Vert}{b-\varphi\circ u}\le \frac{C_m}{\sqrt{\kappa}}(\frac{1}{R}+\frac{1}{\sqrt{T}}+\sqrt{A})\underset {Q_{R,T}}{\sup}(\frac{1}{\cos{\sqrt{\kappa}\rho\circ u}})^2.   
\end{equation*}
\end{theorem}
Firstly, recall the following elementary lemma:
\begin{lemma}[cf. \cite{MR2285258}]\label{thm:auxiliary function}
Let $R, T>0, \alpha\in(0,1)$. Then there is a smooth function $\psi=\psi(r,t):[0,\infty)\times(-\infty, 0]\to[0,1]$, which is supported on $[0,R]\times[-T,0]$, and a constant $C_{\alpha}>0$ depending only on $\alpha$ such that the following hold: \\
(1) $\psi\equiv 1$ on $[0,R/2]\times[-T/4,0]$;\\
(2) $\partial_r\psi\le 0$ on $[0,\infty)\times(-\infty, 0]$, and $\partial_r\psi\equiv 0$ on $[0,R/2]\times(-\infty, 0]$; \\
(3) we have
\begin{equation*}
\frac{\vert\partial_r \psi\vert}{\psi^{\alpha}}\le\frac{C_{\alpha}}{R},\frac{\vert\partial^2_r \psi\vert}{\psi^{\alpha}}\le\frac{C_{\alpha}}{R^2},
\frac{\vert\partial_t \psi\vert}{\psi^{1/2}}\le\frac{C}{T},
\end{equation*}
where $C>0$ is a universal constant.
\end{lemma}
\begin{proof}[Proof of Theorem \ref{thm:gradient estimate}]
We take a function $\psi:[0,\infty)\times(-\infty,0]\to [0,1]$ in Lemma \ref{thm:auxiliary function} with $\alpha=3/4$, and set
\begin{equation*}
\psi(x,t):=\psi(r(x),t).
\end{equation*}
Denote 
\begin{equation*}
\Phi=(\Delta_V-\partial_t)(\psi\omega)-\frac{2g(\nabla\psi, \nabla(\psi\omega))}{\psi}-\frac{2g(\nabla(\psi\omega), \nabla(\varphi\circ u))}{b-\varphi \circ u}.    
\end{equation*}
By direct computations and using Lemma \ref{thm:estimate of u}, we have
\begin{equation}\label{eq:main crude inequality}
\begin{split}
\Phi=&\psi(\Delta_V-\partial_t)\omega-\frac{2\psi g(\nabla\omega, \nabla(\varphi\circ u))}{b-\varphi\circ u}
-\frac{2\omega\Vert\nabla\psi\Vert^2}{\psi} \\
&+\omega(\Delta_V-\partial_t)\psi-\frac{2\omega g(\nabla\psi, \nabla(\varphi\circ u))}{b-\varphi\circ u} \\
\ge &2\kappa(1-b)(b-\varphi\circ u)\psi\omega^2-2A\psi \omega \\
&+\omega(\Delta_V-\partial_t)\psi-\frac{2\omega\Vert\nabla\psi\Vert^2}{\psi}-\frac{2\omega g(\nabla\psi,\nabla(\varphi\circ u))}{b-\varphi\circ u}. \\
\end{split}    
\end{equation}   
Rewrite (\ref{eq:main crude inequality}) as
\begin{equation*}
\begin{split}
2\kappa(1-b)(b-\varphi\circ u)\psi\omega^2\le& 2A\psi \omega-\omega(\Delta_V-\partial_t)\psi \\
&+\frac{2\omega\Vert\nabla\psi\Vert^2}{\psi}+\frac{2\omega g(\nabla\psi,\nabla(\varphi\circ u))}{b-\varphi\circ u}+\Phi.    
\end{split}
\end{equation*}
By the Young inequality and $0\leq \psi\le 1$, for any $\epsilon>0$, we have
\begin{equation*}
\begin{split}
 2A\psi\omega \le  \epsilon \psi^2\omega^2+\frac{A^2}{\epsilon} \le  \epsilon\psi\omega^2+\frac{A^2}{\epsilon}.
\end{split}
\end{equation*}
Lemma \ref{thm:auxiliary function} tells us that
\begin{equation}\label{eq:3/4 parameter}
\frac{\vert\partial_r \psi\vert^2}{\psi^{3/2}}\le\frac{C^2_{3/4}}{R^2} \quad \text{and}\quad
\frac{\vert\partial^2_r \psi\vert^2}{\psi^{3/2}}\le\frac{C^2_{3/4}}{R^4}.   
\end{equation}
According to \cite{MR2995205}, we know
\begin{equation}\label{eq:Laplacian of Riemannian distance}
\begin{split}
\Delta_V r &\le \sqrt{(m-1)A}\,\text{coth}(\sqrt{\frac{A}{m-1}}r)+v(r) \\
&\le(\sqrt{(m-1)A}+\frac{m-1}{r})+v(r)
\end{split}
\end{equation}
where $m$ is the dimension of $M$.
Using (\ref{eq:3/4 parameter}) and (\ref{eq:Laplacian of Riemannian distance}), we get at any $x\not\in \text{Cut}(x_0)$ 
\begin{equation*}
\begin{split}
-\omega(\Delta_V-\partial_t)\psi\le & -\omega (\partial_r\psi(\Delta_V-\partial_t)r+\partial_r^2\psi\Vert\nabla r\Vert^2-\partial_t\psi) \\
\le & \omega (\vert\partial_r\psi\vert(\sqrt{(m-1)A}+\frac{m-1}{R}+v(R))+\vert\partial_r^2\psi\vert+\vert\partial_t\psi\vert) \\
\le & (\epsilon\psi\omega^2+\frac{(\sqrt{(m-1)A}+\frac{m-1}{R}+v(R))^2}{4}\frac{\vert\partial_r\psi\vert^2}{\epsilon\psi}) \\
&+(\epsilon\psi\omega^2+\frac{1}{4}\frac{\vert\partial_r^2\psi\vert^2}{\epsilon\psi})+(\epsilon\psi\omega^2 +\frac{1}{4}\frac{\vert\partial_t\psi\vert^2}{\epsilon\psi}) \\
\le & 3\epsilon\psi\omega^2+\frac{C^2_{3/4}(1+3(m-1)^2)}{4\epsilon}\frac{\psi^{1/2}}{R^4}+\frac{3C^{2}_{3/4}v^2(R)}{4\epsilon}\frac{\psi^{1/2}}{R^2} \\
&+\frac{C^2}{4\epsilon}\frac{1}{T^2}+\frac{3C^{2}_{3/4}}{4\epsilon}\frac{(m-1)A\psi^{1/2}}{R^2} \\
\le & 3\epsilon\psi\omega^2+\frac{C^2_{3/4}(1+3(m-1)^2)}{4\epsilon}\frac{1}{R^4}+\frac{3C^{2}_{3/4}v^2(R)}{4\epsilon}\frac{1}{R^2} \\
&+\frac{C^2}{4\epsilon}\frac{1}{T^2}+\frac{3C^{2}_{3/4}}{4\epsilon}\frac{(m-1)A}{R^2},
\end{split}
\end{equation*}
Due to the Young inequality, the following two inequalities hold respectively
\begin{equation*}
\begin{split}
\frac{2\omega\Vert\nabla\psi\Vert^2}{\psi}\le & \epsilon\psi\omega^2+\frac{\Vert\nabla\psi\Vert^4}{\epsilon\psi^3} \\
\le & \epsilon\psi\omega^2+\frac{C^4_{3/4}}{\epsilon}\frac{1}{R^4},  
\end{split}
\end{equation*}
\begin{equation*}
\begin{split}
\frac{2\omega g(\nabla\psi,\nabla(\varphi\circ u))}{b-\varphi\circ u} \le & \frac{2\omega \Vert\nabla\psi\Vert \Vert\nabla(\varphi\circ u)\Vert}{b-\varphi\circ u} \\
\le & 2\sqrt{\kappa}\omega^{3/2}\Vert\nabla\psi\Vert \\
\le & \frac{3\epsilon}{4}\psi\omega^2+\frac{4\kappa^2}{\epsilon^3}\frac{\Vert\nabla\psi\Vert^4}{\psi^3} \\
\le & \frac{3\epsilon}{4}\psi\omega^2+\frac{4\kappa^2C^{4}_{3/4}}{\epsilon^3}\frac{1}{R^4}.
\end{split}
\end{equation*}
The above estimates lead to
\begin{equation}\label{eq:estimate of omega}
\begin{split}
2\kappa(1-b)(b-\varphi\circ u)\psi\omega^2\le& \frac{23\epsilon}{4}\psi\omega^2+\frac{C^2}{4\epsilon}\frac{1}{T^2}+\frac{3C^{2}_{3/4}v^2(R)}{4\epsilon}\frac{1}{R^2} \\
&+\frac{C^2_{3/4}}{\epsilon}(\frac{1+3(m-1)^2}{4}+C^2_{3/4} +\frac{4\kappa^2C^{2}_{3/4}}{\epsilon^2})\frac{1}{R^4} \\
&+\frac{A^2}{\epsilon}+\frac{3C^{2}_{3/4}}{4\epsilon}\frac{(m-1)A}{R^2}+\Phi. 
\end{split}
\end{equation}   
at every point in $Q_{R,T}$, where the universal constants
$C_{3/4},C>0$ are given in Lemma \ref{thm:auxiliary function}.

Suppose that the maximum of $\psi\omega$ in $Q_{R,T}$ is reached at $(\bar{x},\bar{t})$. We may assume that $\psi\omega$ is positive at $(\bar{x},\bar{t})$, or else the theorem follows trivially. Then at point $(\bar{x},\bar{t})$, we have
\begin{equation*}
\nabla(\psi\omega)=0, \quad \Delta_V(\psi\omega)\le 0,\quad \partial_t(\psi\omega)\ge 0.
\end{equation*}
Thus 
\begin{equation}\label{eq:maximum point condition}
\Phi(\bar{x},\bar{t})=(\Delta_V-\partial_t)(\psi\omega)-\frac{2g(\nabla\psi, \nabla(\psi\omega))}{\psi}-\frac{2g(\nabla(\psi\omega), \nabla(\varphi\circ u))}{b-\varphi \circ u}\le 0
\end{equation}
at point $(\bar{x},\bar{t})$.
Set
\begin{equation*}
\delta=(1-b)(b-\varphi(u(\bar{x},\bar{\tau}))).   
\end{equation*}
Note that $\delta\in (0,\frac{1}{4}]$. In fact, from \eqref{2.1.}, we obtain that $0\leq \varphi<1$ and $\frac{1}{2}\leq b<1$, and thus $\delta>0$. By \eqref{2.1.} and the inequality $ab\leq \left(\frac{a+b}{2}\right)^2$ for any $a,b>0$, we get $\delta\leq \frac{(\cos\sqrt{k}\rho\circ u)^2}{4}\leq \frac{1}{4}$. Since
\begin{equation*}
\frac{1}{1-b}=2\underset{Q_{R,T}}{\text{sup}}\frac{1}{\cos\sqrt{\kappa}\rho\circ u}
\end{equation*}
and
\begin{equation*}
\frac{1}{b-\varphi(u(\bar{x},\bar{\tau}))}\le \frac{2}{1-\text{sup}_{Q_{R,T}}\varphi\circ u}=2\underset{Q_{R,T}}{\text{sup}}\frac{1}{\cos\sqrt{\kappa}\rho\circ u},   
\end{equation*}
it follows that
\begin{equation*}
\frac{1}{\delta}\le 4\underset{Q_{R,T}}{\text{sup}}(\frac{1}{\cos\sqrt{\kappa}\rho\circ u})^2.   
\end{equation*}
Setting $\epsilon = 4\kappa\delta/23$ in (\ref{eq:estimate of omega}), combining with (\ref{eq:maximum point condition}), we get
\begin{equation*}
\begin{split}
\kappa\delta\psi\omega^2\le & \frac{23C^2_{3/4}}{4\kappa\delta}(\frac{1+3(m-1)^2}{4}+C^2_{3/4} +\frac{529C^{2}_{3/4}}{4\delta^2})\frac{1}{R^4} \\
&+\frac{23C^2}{16\kappa\delta}\frac{1}{T^2}+\frac{69C^2_{3/4}v^2(R)}{16\kappa\delta}\frac{1}{R^2}+\frac{23}{4\kappa\delta}A^2+\frac{69C^{2}_{3/4}}{16\kappa\delta}\frac{(m-1)A}{R^2} 
\end{split}  
\end{equation*}
Hence
\begin{equation*}
\begin{split}
(\psi\omega)^2 \le & \frac{23C^2_{3/4}}{4\kappa^2\delta^2}(\frac{1+3(m-1)^2}{4}+C^2_{3/4} +\frac{529C^{2}_{3/4}}{4\delta^2})\frac{1}{R^4} \\
&+\frac{23C^2}{16\kappa^2\delta^2}\frac{1}{T^2}+\frac{69C^2_{3/4}v^2(R)}{16\kappa^2\delta^2}\frac{1}{R^2}+\frac{23}{4\kappa^2\delta^2}A^2+\frac{69C^{2}_{3/4}}{16\kappa^2\delta^2}\frac{(m-1)A}{R^2} \\
\le & \frac{23C^2_{3/4}}{4\kappa^2\delta^2}(\frac{1+6(m-1)^2}{4}+C^2_{3/4} +\frac{529C^{2}_{3/4}}{4\delta^2})\frac{1}{R^4} \\
&+\frac{23C^2}{16\kappa^2\delta^2}\frac{1}{T^2}+\frac{69C^2_{3/4}C_V}{16\kappa^2\delta^2}\frac{1}{R^4}+\frac{92+69C^{2}_{3/4}}{16\kappa^2\delta^2}A^2
\end{split}  
\end{equation*}
Owing to $\delta\in(0,1)$, there is a positive constant $\bar{c}_m$ depending on $m$ such that
\begin{equation*}
(\psi\omega)^2(\bar{x},\bar{t})\le \frac{\bar{c}_m}{\kappa^2\delta^4}(\frac{1}{R^4}+\frac{1}{T^2}+A^2). 
\end{equation*}
It follows that, for all $(x,t)\in Q_{R,T}$,
\begin{equation*}
\psi\omega(x,t)\le \psi\omega(\bar{x},\bar{t})\le \frac{\bar{c}_m^{1/2}}{\kappa\delta^2}(\frac{1}{R^2}+\frac{1}{T}+A). 
\end{equation*}
Using $\psi\equiv 1$ on $Q_{R/2,T/4}$, we know
\begin{equation*}
\begin{split}
\frac{\Vert du\Vert}{b-\varphi\circ u} &\le \frac{\bar{c}_m^{1/4}}{\sqrt{\kappa}\delta}(\frac{1}{R}+\frac{1}{\sqrt{T}}+\sqrt{A}) \\
&\le \frac{4\bar{c}_m^{1/4}}{\sqrt{\kappa}}(\frac{1}{R}+\frac{1}{\sqrt{T}}+\sqrt{A}) \underset{Q_{R,T}}{\text{sup}}(\frac{1}{\cos\sqrt{\kappa}\rho\circ u})^2    
\end{split}
\end{equation*}
on $Q_{R/2,T/4}$. This finishes the proof. 
\end{proof}

We are now in a position to show Theorem \ref{thm:Liouville Theorem heat flow}.
\begin{proof}[Proof of Theorem \ref{thm:Liouville Theorem heat flow}]
Let $u:M\to N$ be an ancient solution to V-harmonic heat flow. For $R>0$ we put 
\begin{equation*}
\mathcal{A}_{R}:= \underset{Q_{R,R^2}}{\text{sup}} (\frac{1}{\cos\sqrt{\kappa}\rho\circ u})^2.  
\end{equation*}
From the growth condition, we know that $\mathcal{A}_{R}=o(R)$ as $R\to \infty$. For any fixed point $(x,t)\in M\times(-\infty,0)$ and a sufficiently large $R>0$, due to Theorem \ref{thm:gradient estimate}, we know when $A=0$,
\begin{equation*}
\frac{\Vert du\Vert}{b}\le \frac{\Vert du\Vert}{b-\varphi\circ u}\le \frac{2c_m \mathcal{A}_{R}}{R}    
\end{equation*}
at $(x,t)$. Note that $b\le 1$. We complete the proof by letting $R\to \infty$.
\end{proof}

\section*{Acknowledgements}

\end{document}